\documentclass[letterpaper,10pt,reqno]{amsart}

\usepackage{graphicx}
\usepackage{caption}
\usepackage{subcaption}

\usepackage{indentfirst} 
\usepackage{mathtools} 
\usepackage{mathabx} 
\usepackage{amsthm}
\usepackage{thmtools}
\usepackage{enumitem} 
\usepackage{fancyhdr}
\usepackage{lipsum}

\usepackage[colorlinks=true]{hyperref} 
\usepackage[usenames,dvipsnames]{xcolor} 
\usepackage{ifthen}
\usepackage{tikz}
\usetikzlibrary{decorations.pathreplacing}
\usepackage{tikz-cd}
  
\makeatletter

\def\MRbibitem{\@ifnextchar[\my@lbibitem\my@bibitem}

\def\mybiblabel#1#2{\@biblabel{{\hyperref{http://www.ams.org/mathscinet-getitem?mr=#1}{}{}{#2}}}}

\def\myhyperanchor#1{\Hy@raisedlink{\hyper@anchorstart{cite.#1}\hyper@anchorend}}

\def\my@lbibitem[#1]#2#3#4\par{%
  \item[\mybiblabel{#2}{#1}\myhyperanchor{#3}\hfill]#4%
  \@ifundefined{ifbackrefparscan}{}{\BR@backref{#3}}%
  \if@filesw{\let\protect\noexpand\immediate
    \write\@auxout{\string\bibcite{#3}{#1}}}\fi\ignorespaces%
}

\def\my@bibitem#1#2#3\par{%
  \refstepcounter\@listctr
  \item[\mybiblabel{#1}{\the\value\@listctr}\myhyperanchor{#2}\hfill]#3%
  \@ifundefined{ifbackrefparscan}{}{\BR@backref{#2}}%
  \if@filesw\immediate\write\@auxout
    {\string\bibcite{#2}{\the\value\@listctr}}\fi\ignorespaces%
}

\makeatother

\DeclareFontFamily{U} {MnSymbolA}{}
\DeclareFontShape{U}{MnSymbolA}{m}{n}{
   <-6> MnSymbolA5
   <6-7> MnSymbolA6
   <7-8> MnSymbolA7
   <8-9> MnSymbolA8
   <9-10> MnSymbolA9
   <10-12> MnSymbolA10
   <12-> MnSymbolA12}{}
\DeclareFontShape{U}{MnSymbolA}{b}{n}{
   <-6> MnSymbolA-Bold5
   <6-7> MnSymbolA-Bold6
   <7-8> MnSymbolA-Bold7
   <8-9> MnSymbolA-Bold8
   <9-10> MnSymbolA-Bold9
   <10-12> MnSymbolA-Bold10
   <12-> MnSymbolA-Bold12}{}
\DeclareSymbolFont{MnSyA} {U} {MnSymbolA}{m}{n}
 \DeclareFontFamily{U} {MnSymbolC}{}
\DeclareFontShape{U}{MnSymbolC}{m}{n}{
  <-6> MnSymbolC5
  <6-7> MnSymbolC6
  <7-8> MnSymbolC7
  <8-9> MnSymbolC8
  <9-10> MnSymbolC9
  <10-12> MnSymbolC10
  <12-> MnSymbolC12}{}
\DeclareFontShape{U}{MnSymbolC}{b}{n}{
  <-6> MnSymbolC-Bold5
  <6-7> MnSymbolC-Bold6
  <7-8> MnSymbolC-Bold7
  <8-9> MnSymbolC-Bold8
  <9-10> MnSymbolC-Bold9
  <10-12> MnSymbolC-Bold10
  <12-> MnSymbolC-Bold12}{}
\DeclareSymbolFont{MnSyC} {U} {MnSymbolC}{m}{n}

\DeclareMathSymbol{\top}{\mathord}{MnSyA}{219} 
\DeclareMathSymbol{\plus}{\mathord}{MnSyC}{20} 

\declaretheorem[numberwithin=section]{theorem}
\declaretheorem[sibling=theorem]{lemma}
\declaretheorem[sibling=theorem]{corollary}
\declaretheorem[sibling=theorem]{proposition}
\declaretheorem[sibling=theorem,style=definition]{definition}
\declaretheorem[sibling=theorem]{remark}
\declaretheorem[sibling=theorem]{example}

\hypersetup{bookmarksdepth = 3} 
\numberwithin{equation}{section}     

\setlist[enumerate,1]{label={\upshape(\alph*)},ref=\alph*}
\setlist[enumerate,2]{label={\upshape(\arabic*)},ref=\arabic*}

\newcommand{\R}{\mathbb{R}}

\newcommand{\N}{\mathbb{N}}

\def\phi{\varphi}
\def\R{{\mathbb R}}

\def\N{{\mathbb N}}

\def\sm{\setminus}

\newcommand{\vertiii}[1]{{\left\vert\kern-0.25ex\left\vert\kern-0.25ex\left\vert #1 
    \right\vert\kern-0.25ex\right\vert\kern-0.25ex\right\vert}}
\newcommand{\invertiii}[1]{{\vert\kern-0.25ex\vert\kern-0.25ex\vert #1 
    \vert\kern-0.25ex\vert\kern-0.25ex\vert}}

\begin{document}

\title{The Lyapunov spectrum as the Newton-Raphson method for countable Markov interval maps}

\date{\today}



\author{ Nicol\'as Ar\'evalo H.}	
\address{Facultad de Matem\'aticas,
Pontificia Universidad Cat\'olica de Chile (UC), Avenida Vicu\~na Mackenna 4860, Santiago, Chile}
\email{\href{narevalo1@mat.uc.cl}{narevalo1@mat.uc.cl}}

\begin{abstract}
We consider MRL maps (Markov-Renyi-Lüroth), a class of interval maps with infinitely many branches that can have parabolic fixed points. We prove that for every MRL map $T$, the Lyapunov spectrum can be expressed in terms of the Legendre transform of the topological pressure of $-t\log|T'|$, generalizing previous results in the area. We also show that the Lyapunov spectrum coincides with a function directly related to the Newton-Raphson method applied to the topological pressure of $-t\log|T'|$.\end{abstract}

\maketitle

\section{Introduction}
Let $\{I_{n}\}_{n\in \N}$ be a countable collection of intervals with disjoint interiors in $[0,1]$. Let $I=\bigcup_{n\in \N}I_{n}$ and $T:I\rightarrow [0,1]$ be a function differentiable in each $I_{n}$. The \emph{Lyapunov exponent} of $T$ is defined for each $x\in I$ by $$\lambda(x):=\lim_{n\rightarrow \infty}\dfrac{1}{n}\log |(T^{n})'(x)|,$$ whenever the limit exists. It measures the exponential rate of divergence of infinitesimally close orbits. It is one of the fundamental quantities used to describe a dynamical system whose transformation is locally differentiable. A consequence of Birkhoff's ergodic theorem is that if $\mu$ is an ergodic $T$-invariant measure, then the Lyapunov exponent is constant for $\mu$-almost every point. However, since different ergodic measures are mutually singular, they may yield different values for the Lyapunov exponent. For every $\alpha\in  \R$, we define the set \begin{align*}
  J_{\alpha}:=\left\{x\in I: \lambda(x)=\alpha\right\},
\end{align*} whenever the set is non-empty. The sets $\{J_{\alpha}\}_{\alpha\in \R}$ and the collection of points for which the Lyapunov exponent is not defined induce a decomposition of the repeller of $T$ (see Definition \ref{Def31}). We define the Lyapunov spectrum to study this decomposition. For every $\alpha\in\R$ such that $J_{\alpha} \neq \emptyset$,  the \emph{Lyapunov spectrum} of $\alpha$ is defined by
\begin{align*}
    L(\alpha):=\text{Dim}_{H}(J_{\alpha}),
\end{align*} where $\text{Dim}_{H}$ denotes the Hausdorff dimension (see Definition \ref{HausdorffDim}). \\

The systematic study of the Lyapunov spectrum began with the 1999 work of Weiss \cite{we}. Using tools from his joint work with Pesin \cite{yh}, he proved that for conformal expanding maps and Axiom A diffeomorphism, the function $L$ has a bounded domain and is real analytic. This is a remarkable result since each level set is dense in the repeller. Thus, even though the decomposition into level sets is extremely complicated, the function that encodes it is as regular as possible. The convexity properties of the Lyapunov spectrum have been subject of recent interest \cite{ik, jpv, m}. The study of Lyapunov spectrum has also been carried out in realm of non-uniform hyperbolicity. Indeed,  
Pollicott and  Weiss \cite{mh}, and Nakaishi  \cite{na} obtained results describing the Lyapunov spectrum for the Manneville Pomeau map. More generally,  Gelfert and Rams \cite{gr} in 2009 studied a fairly general family of interval maps with finitely many branches having parabolic fixed points. In these works it is shown that $L$ is still defined on a bounded interval but it can have points where it is not differentiable. The Lyapunov spectrum has also been studied for dynamical systems defined  over non-compact spaces. Indeed, the Lyapunov spectrum of the Gauss map has been thoroughly studied. Starting with the work of  Pollicott and Weiss \cite{mh} and completed later by  Kesseb{\"o}hmer and Stratmann \cite{Ks} and by Fan et al. \cite{flww}. They showed that, in this setting, the function $L$ is real analytic and it is defined on an unbounded interval. Finally, in 2010, Iommi \cite{io} studied a class of interval maps with countably many branches having parabolic fixed points. He showed that  $L$ is defined on an unbounded interval and could have points where it is not differentiable.\\

In all the above mentioned cases, a relation between the Lyapunov spectrum with the Legendre transform of the topological pressure $P$ for the potential $-t\log{|T'|}$ is established (see Definitions \ref{Legendre transform } and \ref{TopologicalP} respectively). More precisely,  for all $\alpha\neq 0$ such that $J_{\alpha}\neq \emptyset $, the following holds:  \begin{align}\label{QINTRO}    L(\alpha)=\frac{1}{\alpha}\inf_{t\in \R}(P(-t\log|T'|)+t\alpha).\end{align}
In this article, we extend this result so as to include a class of interval maps having countably many branches and, possibly, parabolic fixed points, the so-called MRL maps (see Theorem \ref{THM 5.1}). Among MRL maps, there are functions $T$ with parabolic fixed points and whose topological pressure with respect to $-t\log|T'|$ exhibits new behaviour. Indeed, there exists $0\leq t_{\infty}<\text{Dim}_{H}(\Lambda)$ where $\Lambda$ denotes the repeller of $T$ (see Definition \ref{Def31}), such that the pressure is infinity for $t<t_{\infty}$, is equal to zero for $t\geq \text{Dim}_{H}(\Lambda)$, is real analytic in $(t_{\infty},\text{Dim}_{H}(\Lambda))$ and the lateral limit of the pressure at $t_{\infty}$ is finite. Examples exhibiting this behaviour are provided (see Example \ref{Example 4.6}). As far as we know, these types of examples are new. Also for MRL maps, in Theorem \ref{Thm 1} we describe the properties of the topological pressure of $-t\log|T'|$ and its thermodynamic formalism, generalizing previous results of Markov interval maps.

In 2010, Iommi \cite{io2} noted that for a class of uniformly expanding interval maps with finitely many branches, the corresponding Lyapunov spectrum is nothing but the Newton map given by the Newton-Raphson method applied to the topological pressure of $-t\log|T'|$ as a function of $t\in \R$. In this article we generalize this result for MRL maps. For that, we define the S-Newton-Raphson map as a generalization of the Newton map (see Definition \ref{S-Newton}). Let $Dom(L)=\{\alpha\in \R: J_{\alpha}\neq \emptyset\}$, i.e., $Dom(L)$ denotes the domain of $L$. We prove the following. 
\begin{theorem}\label{PrincipalT}
 Let T be an MRL map. Then for every $\alpha\in Dom(L)$
 \begin{align*}
    L(\alpha)=
    Ns_{P}(\alpha),
\end{align*}where $Ns_{P}$ is the S-Newton-Raphson map applied to the topological pressure of $-t\log|T'|$.
\end{theorem}
As in the case studied in \cite{io2}, the techniques used in the proof of Theorem \ref{PrincipalT} stem from thermodynamic formalism. In particular, in a detailed study of the topological pressure map $t \mapsto P(-t\log |T'|)$. However, contrary to what happens in \cite{io2}, the pressure function can have points where it is not defined or where it is not differentiable, the so-called phase transitions. Moreover, the  Newton-Raphson method needs to be adapted to this not necessarily finite nor differentiable setting.\\

In summary, the difficulties that need to be addressed are the non-uniform hyperbolicity of the map $T$ (it can have parabolic fixed points) and the fact that it is defined in a non-compact space (creating a great deal of convergence difficulties). Both of these properties reflect on the regularity of the pressure function and yield complications implementing the Newton-Raphson method. In section 2, we define Slide maps. These functions include pressure maps $t\mapsto P(-t\log|T'|)$ of MRL maps. We define the Legendre transforms and a generalization of the Newton-Raphson method for slide maps. Section 3 defines MRL maps, we present several examples. In section 4, we study the thermodynamical formalism of MRL maps. Finally, in section 5, we present the proof of our main result.

\section{The Newton-Raphson method and the Legendre transform}

In this section we define the slide functions, these are a family of real-valued functions which share some properties with the topological pressure of countable Markov interval maps. We also define the Legendre transform and the S-Newton-Raphson map for slide functions. Moreover, at the end of this section, an essential relation between the Legendre transform and the S-Newton-Raphson map will be established. 
\begin{definition}\label{Defslide}
Let $t_{\infty}$ and $d$ be non-negative real numbers such that $t_{\infty}<d$. A non-increasing function $f:\R \rightarrow \R \cup \{\infty\}$ is said to be \textit{slide} if $f(t)=\infty$ whenever $t<t_{\infty}$, has a zero at $t=d$, and, either
\begin{itemize}
    \item the function $f$ is $C^{2}$, strictly decreasing and strictly convex in the interval $(t_{\infty},d)$, and $f(t)=0$ whenever $t\geq d$ in which case we will call it a \textit{parabolic slide function}, or
    \item the function $f$ is $C^{2}$, strictly decreasing and strictly convex in the interval $(t_{\infty},\infty)$, in which case we will call it a \textit{non-parabolic slide function}.
\end{itemize}We say that a slide function $f$ is of \textit{continuous type} whenever $\lim_{t\rightarrow t^{+}_{\infty}}f(t)=\infty$, or, of \textit{discontinuous type} whenever $\lim_{t\rightarrow t^{+}_{\infty}}f(t)<\infty$. Let  $a_{f}:=\lim_{t\rightarrow t_{\infty}^{+}}f'(t)$ and 
\begin{align}\label{Qlim}
    b_{f}:=\begin{cases}\lim_{t\rightarrow d^{-}}f'(t) & \text{if $f$ is a parabolic slide function.}\\
    \lim_{t\rightarrow \infty} f'(t) & \text{if $f$ is a non-parabolic slide function.} \end{cases}
\end{align}      
\end{definition}
 Figure \ref{fig:my_label2} shows the graphs of different classes of slide functions.
 
\begin{figure}
    \centering
    \includegraphics[width=0.7\linewidth]{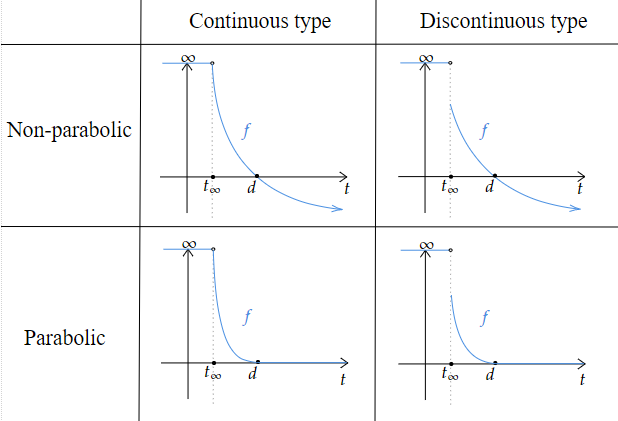}
    \caption{Types of slide functions.}
    \label{fig:my_label2}
\end{figure}

\subsection{The Newton-Raphson method}\hfill \\
The Newton map of a slide function $f$ is defined for every $t\geq  t_{\infty}$ by $N_{f}(t):=t-\frac{f(t)}{f'(t)}$ whenever $f'(t)\neq 0$. This is a root estimation of $f$ given by the line tangent to its graph at $(t,f(t))$. Also, when it exists, the orbit of every point $t$ under the Newton map converges quadratically to the root of $f$. This algorithmic method that approximates the roots of a real-valued differentiable function is known as the Newton-Raphson method.\\ \newline 
There are two problems in applying the Newton-Raphson method to slide functions: when $t=t_{\infty}$ and when $t=d$. In the first case, whenever $f$ is of discontinuous type and $a_{f}>-\infty$, we can have different linear approximations by tangent lines with slopes less than $a_{f}$. Similarly, if $t=d$, $f$ is a parabolic slide function and $b_{f}<0$, then $f$ is not differentiable at $d$. For these two problems, we will consider a generalization of the Newton map.
\begin{definition}\label{DefSupp}
     Let $\alpha$ and $c$ be real numbers. We say that a line $L(t)=\alpha t+c$ is a \textit{support line} of $f$ whenever the graph of $L$ lies below the graph of $f$, i.e., for all $t\in \R$ we have $L(t)\leq f(t)$.
\end{definition}
Adding a bit of notation, when $f$ is a slide function of discontinuous type, we will say that a line is tangent to the graph of $f$ at $(t_{\infty},\lim_{t\rightarrow t_{\infty}^{+}}f(t))$ if it contains this point and does not intersect any other point of the graph of $f$.
\begin{proposition}\label{Support}
     Let $f$ be a slide function and $\alpha\in \R$.
     \begin{itemize}
         \item [(i)] If $-\alpha<b_{f}$, then there exists a support line $L_{\alpha}$ with slope $-\alpha$ tangent to the graph of $f$.
         \item [(ii)] If $f$ is a parabolic slide function and $b_{f}\leq -\alpha<0$, then there exists a support line $L_{\alpha}$ with slope $-\alpha$ tangent to the graph of $f$.
     \end{itemize}
\end{proposition}
\begin{proof}
    Let $\alpha\in \R$ be such that $-\alpha\in (a_{f},b_{f})$. Since $f'$ is strictly increasing, it is injective. Then, there is a unique $t_{\alpha}>t_{\infty}$ such that $f'(t_{\alpha})=-\alpha$. Thus, the line $$L(t)=-\alpha(t-t_{\alpha})+f(t_{\alpha})$$ is tangent to the graph of $f$, and due to the convexity of $f$, it is a support line. 
    
    Let $f$ be a slide function of discontinuous type with $a_{f}>-\infty$. Let $\alpha\in \R$ be such that $-\alpha\leq a_{f}$. On one hand, for all $t< t_{\infty}$ we have $f(t)=\infty$, and, on the other hand $f'(t)>-\alpha$ whenever $t>t_{\infty}$. Thus, by the convexity of $f$ we have that the line $$L(t)=-\alpha(t-t_{\infty})+\lim_{t\rightarrow t_{\infty}^{+}}f(t)$$ is a support line tangent to the graph of $f$ at $(t_{\infty},\lim_{t\rightarrow t_{\infty}^{+}}f(t))$ since the slope is less than the subsequent decrease of the function. Finally, let $f$ be a parabolic slide function with $b_{f}<0$. Let $\alpha\in \R$ be such that $b_{f}\leq -\alpha<0$. By the convexity of $f$, since for all $t\geq d$ we have $f(t)=0$, and for all $t\in(t_{\infty},d)$ we get $f'(t)<-\alpha$, the line $L(t)=-\alpha(t-d)$ is a support line tangent to the graph of $f$ at $(d,0)$.
\end{proof}
\begin{definition}\label{S-Newton}
    Let $f$ be a slide function and $\alpha\in \R$.
     \begin{itemize}
         \item [(i)] If $-\alpha\in (a_{f},b_{f})$, then we define the \textit{S-Newton-Raphson} map of $\alpha$ by
     \begin{align*}
     Ns_{f}(\alpha):=t_{\alpha}+\dfrac{f(t_{\alpha})}{\alpha},
     \end{align*} where $t_{\alpha}$ is such that $f'(t_{\alpha})=-\alpha$.
         \item [(ii)]  If $f$ is a slide function of discontinuous type and $-\infty<-\alpha \leq a_{f}$, then we define the \textit{S-Newton-Raphson} map of $\alpha$ by 
         \begin{align*}   Ns_{f}(\alpha):=t_{\infty}+\dfrac{ \lim_{t\rightarrow t_{\infty}^{+}}f(t)}{\alpha} .
         \end{align*}
         \item [(iii)] If $f$ is a parabolic slide function and $b_{f}\leq -\alpha<0$, then we define the \textit{S-Newton-Raphson} map of $\alpha$ by $Ns_{f}(\alpha):=d$.
     \end{itemize}
\end{definition}
    Note that as a consequence of Proposition \ref{Support}, the image of the S-Newton-Raphson map is nothing more than the intersection with the $x$-axis of a support line tangent to the graph of $f$ with slope $-\alpha$. Thus,  with this definition, we can avoid the problems the Newton-Raphson method has at $t=t_{\infty}$ and $t=d$. Also, if $\alpha$ is such that $-\alpha\in (a_{f},b_{f})$, then \begin{align}\label{GEN}
        Ns_{f}(\alpha)=N_{f}\circ (f')^{-1}(-\alpha).
    \end{align} This equation shows $Ns_{f}$ as a generalization of $N_{f}$ since both coincide when $f$ is differentiable.
\subsection{The Legendre transform}
\begin{definition}\label{Legendre transform }
    Let $f$ be a slide function and $\alpha\in \R$. The \textit{Legendre transform} of $\alpha$ with respect to $f$ is defined as the infimum between the $y$-axis intercept of lines with slope $-\alpha$ and secant to the graph of $f$, i.e., the Legendre transform of $\alpha$ with respect to $f$ is given by $$F(\alpha)=\inf_{t\in \R}\{f(t)+\alpha t \}.$$
\end{definition}
The following proposition will be essential in the proof of Theorem \ref{THM 5.1}.

\begin{proposition}\label{Lemma 1}
Let $f$ be a slide function and $\alpha\in \R$. Then 
\begin{align*}
    F(\alpha)=\begin{cases}
    f(t_{\alpha})+\alpha t_{\alpha} & \text{if $-\alpha \in (a_{f}, b_{f})$.}\\
    \displaystyle\lim_{t\rightarrow t_{\infty}^{+}}f(t)+\alpha t_{\infty} & \text{if $f$ is of discontinuous type with $-\infty<-\alpha\leq a_{f}$.}\\
    \alpha d & \text{if $f$ is a parabolic slide function and  $b_{f}\leq -\alpha\leq 0$.}\\
    \displaystyle\lim_{t\rightarrow \infty} f(t)+\alpha t & \text{if $f$ is a non-parabolic slide function and $b_{f}\leq -\alpha\leq 0$.}\\
    -\infty & \text{if } 0<-\alpha,
    \end{cases}
\end{align*} where $t_{\alpha} \in \R$ is such that $f'(t_{\alpha})=-\alpha$.
\end{proposition}
\begin{proof}For each $\alpha\in \R$ consider the function $\widehat{F}_{\alpha}(t):=f(t)+\alpha t$. Note that the Legendre transform of $\alpha$ is obtained by considering the infimum of $\widehat{F}_{\alpha}$ over the variable $t$. For all $t\in (t_{\infty},\infty)\sm \{d\}$, we have $\widehat{F}_{\alpha}'(t)=f'(t)+\alpha$.\\

Let $-\alpha< b_{f}$. If $-\alpha\in (a_{f},b_{f})$, then by the Mean Value Theorem,  $\widehat{F}_{\alpha}$ attains its minimum at a point $t_{\alpha}> t_{\infty}$. Therefore, $\inf_{t>t_{\infty}}\widehat{F}_{\alpha}(t)=f(t_{\alpha})+\alpha t_{\alpha}$. On the other hand, if we have $ -\infty< -\alpha\leq a_{f}$, then $\widehat{F}_{\alpha}$ is strictly increasing. Thus, $\inf_{t>t_{\infty}}\widehat{F}_{\alpha}(t)=
    \lim_{t\rightarrow t_{\infty}^{+}}f(t)+\alpha t_{\infty} .$\\ \newline
    Now, let $b_{f} \leq -\alpha <0$. Then, for all $t\in (t_{\infty},d)$ we have $\widehat{F}_{\alpha}'(t)< -b_{f}+f'(t)\leq 0$. Suppose $f$ is a parabolic slide function. Therefore, on one hand $\widehat{F}_{\alpha}$ is non-increasing in $(t_{\infty},d)$ and in the other hand $\widehat{F}_{\alpha}(t)=\alpha t$ whenever $t\geq d$. Thus, $$\inf_{t\in \R}\widehat{F}_{\alpha}(t)=\min\{\lim_{t\rightarrow d}\widehat{F}_{\alpha}(t),\inf_{t>d}\alpha t\}=\alpha d.$$ Now suppose $f$ is a non-parabolic slide function, then for all $t>t_{\infty}$ function $\widehat{F}_{\alpha}$ is non-increasing, since $b_{f}\leq -\alpha< 0$. We obtain that $$\inf_{t\in \R}\widehat{F}_{\alpha}(t)=\lim_{t\rightarrow\infty}\widehat{F}_{\alpha}(t).$$ Finally, let $0<-\alpha$. Since $\widehat{F}_{\alpha}'(t)=f'(t)+\alpha<0$, we have that $\widehat{F}_{\alpha}$ has no critical points and is strictly decreasing in $(t_{\infty},\infty)\sm \{d\}$. Thus, $$\inf_{t\in \R}\widehat{F}_{\alpha}(t)=\lim_{t\rightarrow \infty}\widehat{F}_{\alpha}(t)\leq \lim_{t\rightarrow \infty}\alpha t=-\infty.$$
\end{proof}

 Proposition \ref{Lemma 1} allows us to relate the Legendre transform with the S-Newton-Raphson map of slide functions. The following generalizes \cite[Lemma 2.1]{io2}.
\begin{corollary}\label{Lemma2}
Let $f$ be a slide function and $\alpha\in \R$. If either $-\alpha< b_{f}$ or $f$ is a parabolic slide function with $b_{f}\leq -\alpha < 0$, then
\begin{align*}
    F(\alpha)=\alpha Ns_{f}(\alpha).
\end{align*}
\end{corollary}

\begin{proof}This follows from Definition \ref{S-Newton} together with Proposition \ref{Lemma 1}. If $-\alpha\in (a_{f},b_{f})$, then
\begin{align}\label{Cor1eq}
    \alpha Ns_{f}(\alpha)=\alpha\left(t_{\alpha}+\dfrac{f(t_{\alpha})}{\alpha}\right)= t_{\alpha}\alpha+f(t_{\alpha})=F(\alpha).
\end{align}Similarly, if $-\infty<-\alpha\leq a_{f}$, then replacing $t_{\alpha}$ by $t_{\infty}$ and $f(t_{\alpha})$ by $\lim_{t\rightarrow t_{\infty}^{+}}f(t)$ in equation \eqref{Cor1eq} we get $\alpha Ns_{f}(\alpha)=F(\alpha)$. On the other hand, if $f$ is a parabolic slide function and $b_{f}\leq -\alpha <0$, then $\alpha Ns_{f}(\alpha)=\alpha d=F(\alpha)$.
\end{proof}

\section{Countable Markov interval maps}
In this section we define the class of Markov interval maps we will deal with, the Markov-Renyi-Luröth maps (MRL maps for short). Several examples will be provided.

\begin{definition}\label{Def31}Let $\{I_{n}\}_{n\in \N}$ be a collection of intervals in $[0,1]$ with disjoint interior. Let $I =\bigcup_{n\in \N}I_{n}$ and $T : I\rightarrow [0,1]$ be a function such that for every $n \in \N$ we have $\overline{T(I_{n})}=[0,1]$, where $\overline{T(I_{n})}$ is the closure of $T(I_{n})$. For each array of positive integers $(a_{n})_{n=1}^{m}$ we define the \emph{cylinder} $$I_{(a_{n})_{n=1}^{m}}:=\bigcap^{m}_{r=0}T^{-r}I_{a_{r}},$$ and the \textit{repeller} of $T$ by $$\Lambda:=\bigcap^{\infty}_{n=0}T^{-n}I.$$
The function $T$ is said to be a \textit{countable Markov interval map} if \begin{enumerate}

    \item $T$ is of class $C^{1+\varepsilon}$  (differentiable with Hölder derivative) and monotone on every interval  $I_{n}$.
    \item There is a non-negative integer $m$ and a fixed point $p\in I$ such that for each $x\in I\setminus \{p\}$ we have $|(T^{m})'(x)|>1$ and $|T'(p)|\geq 1$.
    \item \label{Tempered distortion property} $T$ has \textit{tempered distortion property} (\cite{gr2}), i.e., there exists a sequence of positive numbers $(\rho_{m})_{m\in \N}$ that converges to zero such that for every positive integer $n$ 
\begin{align*}
    \sup_{a_{0},...,a_{n}\in \N}\left(\sup_{x,y\in I_{(a_{k})_{k=1}^{n}}}\dfrac{|(T^{n})'(x)| }{|(T^{n})'(y)|}\right)\leq \exp{(n\rho_{n})}.
\end{align*}
\end{enumerate}
\end{definition}For each $n\in \N$, the tempered distortion property assures that the variation between the values of $(T^{n})'$ is uniformly bounded for all cylinders of length $n$. A point $p\in I$ such that $T(p)=p$ and $|T'(p)|=1$ is called a \emph{parabolic fixed point}. For $T$ a countable Markov interval map, the system $(T,\Lambda)$ is semi-conjugated to the full shift on infinite symbols $(\sigma,\N^{\N})$ (see \cite{io,sa}). 

\begin{definition}\label{MRLmaps}
We will say that a countable Markov interval map $T$ is an \textit{MRL map} (Markov-Renyi-Lüroth) if it meets either one of the following conditions:
\begin{enumerate}
    \item There are constants $K>1$ and $C>0$ such that for every $n\in \N$ and $x\in I_{n}$ we have $$C^{-1}n^{K}\leq |T'(x)|\leq Cn^{K}.$$
    \item The collection of end points of the intervals in $\{I_{n}\}_{n\in \N}$ accumulate only at one point, and $T$ is linear in every interval of $\{I_{n}\}_{n\in \N}$ except, for at most, a finite number of them. 
\end{enumerate}Maps $T$ satisfying $(a)$ will be called \textit{Markov-Renyi-like} maps. Similarly, maps $T$ satisfying $(b)$ will be called \textit{Lüroth-like} maps.
\end{definition} 

The collection of MRL maps contains both the families of the Markov-Renyi maps \cite{io,ij1} and the $\alpha$-Lüroth maps \cite{KMS}. But it is larger than these two families (see Example \ref{Example 4.6}). We now present examples of MRL maps that, as we will see in sections 4 and 5, display different thermodynamic formalism and Lyapunov spectra. 

\begin{example}\normalfont\label{Example Gauss}Let $G: (0,1]\rightarrow [0,1]$ be defined for each $x\in (0,1]$ by $$G(x):=\dfrac{1}{x}-\left\lfloor \dfrac{1}{x}\right\rfloor,$$ where $\lfloor \cdot \rfloor$ is the floor function. The transformation $G$ is called \textit{Gauss map}. Define for every $n\in \N$ the interval $I_{n}=(\frac{1}{n+1},\frac{1}{n}]$. Therefore, for every $x\in I_{n}$, we have that $G(x)=\frac{1}{x}-n$. The Gauss map is an MRL map without parabolic fixed points and with repeller $\Lambda=(0,1]\sm \mathbb{Q}$. This map is closely related to the continued fraction expansion of real numbers. Recall that every irrational number $x\in (0,1)$ can be written in a unique way as \emph{continued fraction}, 
  \begin{align}\label{continued}
      x= \cfrac{1}{a_{1}+\cfrac{1}{a_{2}+\cfrac{1}{a_{3}+\cdots }}}=[a_{1},a_{2},a_{3}...],
  \end{align} where $a_{1},a_{2},...\in \N$. Indeed in equation \eqref{continued} we have that $a_{n}=\lfloor \frac{1}{G^{n-1}x}\rfloor$. Define $\frac{p_{n}(x)}{q_{n}(x)}=[a_{1},a_{2},...,a_{n}]$ as the $n$-th rational approximation of $x$. It turns out that the Lyapunov exponent of $x$ satisfies (see \cite[Page 160]{mh}) $$\lambda(x)=-\lim_{n\rightarrow \infty}\dfrac{1}{n}\log \left|x-\dfrac{p_{n}(x)}{q_{n}(x)}\right|.$$ Therefore, the Lyapunov exponent measures the exponential speed at which the rational numbers $\{\frac{p_{n}(x)}{q_{n}(x)}\}_{n\in \N}$ approximates $x$.
  
  An ergodic measure of the Gauss map is the \textit{Gauss measure} defined for every Borel set $A\subseteq [0,1]$ by $$\mu_{G}(A)=\frac{1}{\log 2}\int_{A}\dfrac{1}{1+x}dx.$$  The measure $\mu_{G}$ is absolutely continuous with respect to the Lebesgue measure. Thus, as a consequence of Birkhoff's ergodic theorem  (see \cite{yh}) for Lebesgue almost every point $x$,  $$\lambda(x)=\dfrac{\pi^{2}}{6\log 2}.$$
\end{example}

\begin{example}\normalfont \label{Example 2} Let $R:[0,1)\rightarrow  [0,1]$ defined for each $x\in [0,1)$ by  $$R(x)=\dfrac{1}{1-x}-\left\lfloor\dfrac{1}{1-x}\right\rfloor.$$ Define, for each $n\in \N$, the intervals $I_{n}=[\frac{n-1}{n},\frac{n}{n+1})$. Therefore, for every $x\in I_{n}$, we have that $R(x)=\frac{1}{1-x}-n$. This map is an MRL map and has a parabolic fixed point at $x=0$. The transformation $R$ is called \textit{Renyi map} (also backward continued fraction map). It is related to the backward continued fraction expansion of irrational numbers (see \cite{af}). Indeed, for every irrational number $x \in (0,1)$ there exists a sequence of positive integers $\{a_{n}\}_{n\in\N}$ such that $$x=1- \cfrac{1}{a_{1}-\cfrac{1}{a_{2}-\cfrac{1}{a_{3}-\cdots }}}=[a_{1},a_{2},a_{3}...],$$where $a_{n}=\lfloor 1/(1-R^{n-1}x)\rfloor$. There is also an $R$-invariant measure $\mu_{R}$ absolutely continuous with respect to the Lebesgue measure. This measure is defined for every Borel set $A\subseteq [0,1]$ by $$\mu_{R}(A)=\int_{A}\dfrac{1}{x} \, dx.$$ The measure $\mu_{R}$ is infinite and $\sigma$-finite. Unlike the Gauss map, there is no finite and absolutely continuous invariant measure with respect to the Lebesgue measure (see \cite[Section 2.3]{io}).
\end{example}

\begin{example}\normalfont\label{ALurorh} Let $\mathcal{A}=\{I_{n}\}_{n\in \N}$ be a countable partition of $[0,1]$ such that for every $n\in \N$, the set $I_{n}$ is a left open right closed interval. The intervals in $\mathcal{A}$ are ordered from right to left starting with $I_{1}$ and accumulate at the origin. For each $n \in \N$ let $a_{n}=\text{Leb}(I_{n})$ where $\text{Leb}$ denotes the Lebesgue measure over $[0,1]$. Define the $k$-th tail of $\{a_{n}\}_{n\in \N}$ by
$t_{k}=\sum^{\infty}_{n=k}a_{n}$.  The \textit{$\alpha$-Lüroth} map for the partition $\mathcal{A}$ is given by 
\begin{align*}
    L_{\mathcal{A}}(x):=\begin{cases}(t_{n}-x)/a_{n} & \text{If $x\in I_{n}$.}\\
    0 & \text{If $x=0$.}\end{cases}
\end{align*} The map $L_{\mathcal{A}}$ is an MRL map, and the Lebesgue measure is ergodic for it. Similarly to the Gauss and Renyi maps, the $\alpha$-Lüroth map generates a series expansion of a real number related to $\mathcal{A}$. For $x\in [0,1]\sm J$, where $J$ is a countable set, let $\{l_{k}\}_{k\in \N}$ be a sequence of positive integers such that for every $k\in \N$ we have $L^{k-1}_{\mathcal{A}}(x)\in I_{l_{k}}$. Then the $\mathcal{A}$-Lüroth-expansion of $x$ (see \cite{KMS}) is given by \begin{align*}
    x=\sum^{\infty}_{n=1}(-1)^{n-1}\left(\prod_{i<n}a_{l_{i}}\right)t_{l_{n}}=t_{l_{1}}-a_{l_{1}}t_{l_{2}}+a_{l_{1}}a_{l_{2}}t_{l_{3}}+\cdots.
\end{align*} This expansion relates Lebesgue almost every point $x$ on $[0,1]$ (up to a countable set) with an infinite sequence of positive integers $\{l_{k}\}_{k\in\N}$.     \end{example}
\begin{example}\normalfont    \label{Example 3} The Manneville-Pomeau map is a two-branch map defined over $[0,1]$ by $f(x)=x+x^{1+s} \mod 1$, where $\frac{1}{2}\leq s \leq 1$ (see \cite{pm}). This map is one of the first examples used to study non-uniform hyperbolic systems. Indeed, it has a parabolic point at $x=0$. We consider an infinite branch generalization of the Manneville-Pomeau map by keeping the interval $I_{0}$ with the parabolic fixed point and defining it as linear maps over a countable partition $\{I_{n}\}_{n\in \N}$ of $[0,1]\sm I_{0}$. It is defined as follows: Let $c\in \R$ be such that $1=c+c^{1+s}$. For every $n\in \N$ let $a_{n}=Leb(I_{n})$, where $Leb$ denotes the Lebesgue measure over $[0,1]$. We define the \textit{Infinite Manneville-Pomeau map} of $x\in [0,1]$ with respect to $\{I_{n}\}_{n\in \N}$ by \begin{align*}
    T(x)=\begin{cases} x+x^{1+s} & \text{if}\quad  x\in  I_{0} .\\
     \dfrac{1-(t_{n}+x)}{a_{n}} & \text{if} \quad x\in I_{n} .\end{cases},
\end{align*} where $t_{n}$ is the $n$th tail of $\{a_{n}\}_{n\in \N}$. It has a unique parabolic fixed point at $x=0$. Note that, also in this case, $\Lambda=[0,1)\sm J$, where $J$ is a countable set in $[0,1]$. The Infinite Manneville-Pomeau map is an MRL map with an invariant probability measure that is absolutely continuous with respect to the Lebesgue measure (\cite[section 2.2]{io}).\end{example}

\section{Thermodynamic formalism for MRL maps}
In this section, we define the topological pressure of MRL maps using Sarig's theory \cite{sa2} for Countable Markov shifts. We present the thermodynamic formalism of MRL maps in Theorem \ref{Thm 1}. We give an example of an MRL map with a pressure function whose behavior has not been considered before, that is, a slide parabolic function of discontinuous type (see Example \ref{Example 4.6}). We also provide the definition of the Hausdorff dimension of subsets in $[0,1]$.

\subsection{Hausdorff Dimension}\hfill \\
We briefly recall the definition of Hausdorff dimension, for more details see \cite[Section 2.2]{fa}.
\begin{definition}\label{HausdorffDim}
     Let $F\subseteq [0,1]$. Let $s,\delta\in \R$  be such that $s\geq 0$ and $\delta>0$. We define 
\begin{align*}
    \mathcal{H}^{s}_{\delta}(F):=\inf\left\{\sum_{i}|U_{i}|^{s}:\begin{array}{c}
         \text{ $\{U_{i}\}$ is a countable cover of $F$ such that for every $i$}  \\
         \text{ $|U_{i}|<\delta$.}
    \end{array} \right\}, 
\end{align*} where $|U_{i}|$ denotes diameter of $U_{i}$ with respect the Euclidean metric on $[0,1]$. Also we define $\mathcal{H}^{s}(F)=\lim_{\delta\rightarrow 0}\mathcal{H}^{s}_{\delta}(F)$. The map $s\mapsto\mathcal{H}^{s}(F)$ is non-increasing and satisfies \begin{align}\label{HAUSDORFF}
    \sup\{s:\mathcal{H}^{s}(F)=\infty\}=\inf\{s:\mathcal{H}^{s}(F)=0\}.
\end{align} The Hausdorff dimension of $F$ is defined by the value in equation \eqref{HAUSDORFF} and is denoted by $\text{Dim}_{H}(F)$. 
\end{definition}

\subsection{Topological pressure}
\begin{definition} \label{TopologicalP}
Let $T$ be an MRL map and let $\mathcal{M}_{T}(I)$ be the space of $T$-invariant probability measures.  For each $t\in\R$, the \textit{topological pressure} of $T$ with respect to $-t\log|T'|$ is defined by
\begin{align}\label{QPT}
     P(t):=\sup\left\{h_{\mu}-\int t\log|T'| d\mu: \mu\in\mathcal{M}_{ T}(I) \text{ and }\int \log|T'| d\mu\neq \infty \right\},
\end{align} where $h_{\mu}$ is the metric entropy with respect to the measure $\mu$. An invariant probability measure $\mu$ is said to be an \textit{equilibrium state for $-t\log{|T'|}$} if it reaches the supremum in equation \eqref{QPT}.  
\end{definition}
 For each MRL map $T$, the function $-\log|T'|$ has summable variations (see \cite{sa2}) and $(T,\Lambda)$ is semi-conjugated to the full-shift on a countable alphabet. We can therefore make use of Sarig's results \cite[Corollary 1]{sa3} (see also \cite[Section 3]{MU}) to conclude that the topological pressure satisfies:
\begin{align}\label{PG1}
    P(t):=\lim_{n\rightarrow \infty}\dfrac{1}{n}\log{\sum_{x:T^{n}x=x}\exp{\left(\sum^{n-1}_{k=0}-t\log|T'(T^{k}x)|\right)}}.
\end{align}

\subsection{Thermodynamic formalism for MRL maps}\hfill \\
As we will see, the lack of compactness and the existence of parabolic fixed points are reflected in the properties of the topological pressure of an MRL map. The following theorem shows that the topological pressure function $t\mapsto P(t)$ of an MRL map is a slide function. Without loss
of generality, we will assume that in case the MRL map has a parabolic
fixed point then this is $x=0$. The regularity of $P$ and the fact that it can be a parabolic slide function are proved analogous to \cite[Theorem 3.1]{io}. When $T$ is a Markov-Renyi-like map, the function $ P$ is of continuous type, this follows from \cite[Lemma 6.2]{io}. When $T$ is a Lüroth-like map, the continuity type of $P$ will depend on the collection of intervals where $T$ is linear since $P$ satisfies equation \eqref{PG1}.
\begin{theorem}
\label{Thm 1}
 Let $T$ be an MRL map. 
 Then there exists $t_{\infty}\geq 0$ such that $P$ is a slide function with $d=\text{Dim}_{H}(\Lambda)$. Moreover
 \begin{enumerate}
     \item If $T(0)=0$ and $|T'(0)|=1$, then
     \begin{itemize}
         \item the function $P$ is a parabolic slide function,
         \item  for every $t\in (t_{\infty},\infty)\sm \{\text{Dim}_{H}(\Lambda)\}$, there exists a unique equilibrium state for $-t\log|T'|$. In particular \begin{itemize}
             \item for $t> \text{Dim}_{H}(\Lambda)$ the Dirac delta measure at zero, $\delta_{0}$, is the only equilibrium state for $-t\log|T'|$, and 
             \item for $-\text{Dim}_{H}(\Lambda)\log|T'|$ the measure  $\delta_{0}$ is the unique equilibrium state if and only if the function $P$ is differentiable at $t=\text{Dim}_{H}(\Lambda)$.
         \end{itemize} 
     \end{itemize}
     \item If $T$ has no parabolic fixed points, then 
     \begin{itemize}
         \item the function $P$ is a non-parabolic slide function,
         \item  for every $t\in (t_{\infty},\infty)$, there exists a unique equilibrium state for $-t\log|T'|$. 
     \end{itemize}
     \item If $T$ is a Markov-Renyi-like map, then $P$ is of continuous type.
     \item If $T$ is a Lüroth-like map, let $\{I_{n}\}_{n\in \N}$ be the collection of intervals where $T$ is linear, and for each $n \in \N$ let $a_{n}$ be the Lebesgue measure of $I_{n}$.  Then $P$ is of discontinuous type if and only if $\lim_{t\rightarrow t^{+}_{\infty}}\log(\sum^{\infty}_{n=1}a_{n}^{t})<\infty$.
 \end{enumerate}
\end{theorem}It is worth noting that Theorem \ref{Thm 1} includes results already obtained for Markov-Renyi maps with or without parabolic fixed points \cite[Theorem 3.1]{io}, and also for $\alpha$-Luröth maps since computing the pressure function is analogous to calculating the topological pressure of a countable Markov shift for a locally Hölder potential (see equation \eqref{PG1} and \cite{sa3}).

The Renyi map is a Markov-Renyi-like map with a parabolic fixed point at $x=0$. Therefore its topological pressure function is a parabolic slide function of continuous type where $\text{Dim}_{H}(\Lambda)=1$ and $t_{\infty}=1/2$. Since the Renyi map does not have a finite invariant measure absolutely continuous with respect to the Lebesgue measure, its topological pressure function is differentiable in $t=1$. On the other hand, the Gauss map is a Markov-Renyi-like map without parabolic fixed points. Consequently, it has a non-parabolic slide function of continuous type as its topological pressure function where $\text{Dim}_{H}(\Lambda)=1$ and $t_{\infty}=1/2$. 

The infinite Manneville-Pomeau map is a Lüroth-like map with a parabolic fixed point at $x=0$. So the topological pressure function is a parabolic slide function with $\text{Dim}_{H}(\Lambda)=1$, where the continuity type and the value of $t_{\infty}$ will depend on the intervals where the map is linear (see Example \ref{Example 4.6}). Since it does have a finite invariant measure absolutely continuous with respect to the Lebesgue measure, its topological pressure function is not differentiable in $t=1$. On the other hand, the $\alpha$-Luröth map is a Lüroth-like map without parabolic fixed points. Therefore, its topological pressure function is a non-parabolic slide function with $\text{Dim}_{H}(\Lambda)=1$, where the continuity type and the value of $t_{\infty}$ will depend on the partition where it is defined. We will now show a couple of examples in this regard.
\begin{example}\normalfont \label{Example 4.5}
 Let $m\in\N$ be fixed. Let $C_{m}=\sum_{n=1}^{\infty}(n+10)^{-2}\log^{-2m}(n+10)$. We consider the sequence 
    \begin{align*}
        \{a_{n,m}\}_{n\in \N}=\left\{\dfrac{(n+10)^{-2}\log^{-2m}(n+10)}{C_{m}}\right\}_{n\in \N}.
    \end{align*}This sequence defines a partition of the interval $[0,1]$. Denote by $L_{\{a_{n,m}\}_{n\in \N}}$ the $\alpha$-Luröth map for $\{a_{n,m}\}_{n\in \N}$ (see Example \ref{ALurorh}). By Theorem \ref{Thm 1} and equation \eqref{PG1}, there exists $t_{\infty}\geq 0$ such that for every $t>t_{\infty}$ we have $P_{m}(t)=\log\left(\sum^{\infty}_{n=1}a_{n,m}^{t}\right)$, where $P_{m}(t)$ is the topological pressure of $-t\log |L_{\{a_{n,m}\}}'|$. Thus, for all $t>t_{\infty}$
    \begin{align*}
        P_{m}(t)=\log\left(\dfrac{1}{C_{m}}\sum^{\infty}_{n=1}\dfrac{1}{((n+10)\log^{m}(n+10))^{2t}}\right).
    \end{align*}The series in the parenthesis is a Dirichlet series. These types of series has been used to generate examples of topological pressure functions that are slide functions of discontinuous type (see \cite{ci,KMS}). Indeed, by \cite[Theorem 2.12]{ci} (see also \cite[Theorem 1]{gt}) for $m=1$ and $m=5$, we have that $\lim_{t\rightarrow t_{\infty}^{+}}P_{1}(t)=\infty$ and $\lim_{t\rightarrow t_{\infty}^{+}}P_{5}(t)<\infty$ respectively, where $t_{\infty}=1/2$.
\end{example}
Using example \ref{Example 4.5}, we can construct an MRL map with a pressure function that is a parabolic slide function of discontinuous type. We now construct an infinite Manneville Pomeau map with these properties.
m \begin{example}\normalfont \label{Example 4.6} Let $\frac{1}{2}\leq s\leq 1$ be fixed and let $c\in \R$ be such that $c+c^{1+s}=1$. Let $m\in\N$ be fixed. We define the sequence $\{(1-c)a_{n,m}\}_{n\in \N}$, where $\{a_{n,m}\}_{n\in \N}$ is the same sequence as in Example \ref{Example 4.5}. For every $n\in \N$ define the intervals $I_{n}=[1-t_{n},1-t_{n+1})$, where $t_{n}$ is the corresponding $n$-th tail of $\{(1-c)a_{n,m}\}_{n\in \N}$. The collection $\{I_{n}\}_{n\in \N}$ is a countable partition of the interval $[c,1)$. Let us consider the Infinite Manneville Pomeau map $T_{m}$ with respect to $\{(1-c)a_{n,m}\}_{n\in \N}$. Let $P_{m}(t)$ be the topological pressure of $-t\log|T_{m}'|$. Therefore, by Theorem \ref{Thm 1}, the function $P_{m}$ is a slide function of discontinuous type if and only if $\lim_{t\rightarrow t^{+}_{\infty}}\log\left((1-c)^{t}\sum^{\infty}_{n=1}a_{n,m}^{t}\right)<\infty$. So, by \cite[Theorem 1]{gt} for $m=5$ we have that the series $\sum^{\infty}_{n=1}a_{n,5}^{t}$ converges as $t\rightarrow t_{\infty}^{+}$, where $t_{\infty}=1/2$. Thus, the function $P_{5}$ is a parabolic slide function of discontinuous type, since $T_{5}$ has a parabolic fixed point at $x=0$.
\end{example}
\section{Lyapunov spectrum for MRL maps}
In this section, we establish our main results. The theorem below generalizes results from Iommi \cite[Theorem 4.1]{io}, and, from Kesseböhmer, Munday, and Stratmann \cite[Theorem 2]{KMS}. Indeed Theorem \ref{THM 5.1} also applies to MRL maps such that its topological pressure function is a parabolic slide function of discontinuous type (see Example \ref{Example 5.4}), a case not covered in previous results. Denote by $Dom(L)$ the domain of the Lyapunov spectrum function $L$, i.e., $Dom(L)=\{\alpha\in \R:J_{\alpha}\neq \emptyset\}$.

\begin{theorem}\label{THM 5.1}
Let $T$ be an MRL map. Then $Dom(L)$ is an unbounded subinterval of $[0,\infty)$ and for each nonzero $\alpha\in Dom(L)$
\begin{align}\label{Q51}
    L(\alpha)=\dfrac{1}{\alpha}\inf_{t\in \R}\left\{P(t)+t\alpha\right\}.
\end{align}Whenever $0\in Dom(L)$ we get $L(0)=\lim_{\alpha\rightarrow 0+}L(\alpha)$.
\end{theorem}
\begin{proof} 
We will first use an approximation argument similar to the one employed in \cite[Section 7.1]{io}. For each $n\in \N$, let $T_{n}$ is the restriction of $T$ to $\bigcup^{n}_{k=1}I_{k}$. Let $\Lambda_{n}$ and $L_{n}$ be the repeller and Lyapunov spectrum with respect to $T_{n}$, respectively. Also, let $P_{n}(t)$ be the topological pressure with respect to $-t\log(|T_{n}'|)$ and denote by $Dom(L_{n})$ the domain of $L_{n}$. Since the dynamical system  $(T_{n},\Lambda_{n})$ is defined over a finite collection of intervals, we can deduce  from \cite[Theorem 1]{gr}  that for every $\alpha\in Dom(L_{n})$
\begin{align}\label{Q52}
    L_{n}(\alpha)=\dfrac{1}{\alpha}\inf_{t\in \R}\left\{P_{n}(t)+t\alpha\right\},
\end{align}where $L_{n}(0)=\lim_{\alpha \rightarrow 0^{+}}L(\alpha)$ in case $0\in Dom(L_{n})$. According to the definition of $T_{n}$, every $\alpha$ in $Dom(L_{n})$ is also in $Dom(L_{n+1})$. So, for every $n\in \N$ we have $L_{n}(\alpha)\leq L_{n+1}(\alpha)\leq \cdots \leq L(\alpha)$. It follows that, for each $\alpha\in Dom(L_{n})$ the sequence of numbers $\{L_{k}(\alpha)\}_{k\geq n}$ is monotone and bounded by $\text{Dim}_{H}(\Lambda)$. This implies that the sequence of functions $\{L_{n}\}_{n\in \N}$ has a pointwise limit. By equation \eqref{Q52}, we get for every $\alpha \in Dom(L)$ that 
\begin{align}\label{Q53}
    \lim_{n\rightarrow \infty}\dfrac{1}{\alpha}\inf_{t\in \R}\{P_{n}(t)+t\alpha\}\leq L(\alpha).
\end{align}  Now, we show that $\frac{1}{\alpha}\inf_{t\in \R}\{P(t)+t\alpha\}$ is lower than the left-hand side term of inequality \eqref{Q53}.  Recall, by Definition \ref{Legendre transform }, that $F(\alpha)=\inf_{t\in \R}\{P(t)+t\alpha\}$ and for every $n\in \N$ that $F_{n}(\alpha)=\inf_{t\in \R}\{P_{n}(t)+t\alpha \}$. On the other hand, it is known from Sarig's theory on countable Markov shifts \cite[Theorem 2]{sa2} that for all $t\in \R$, 
\begin{align}\label{SarigConvergence}
    \lim_{n\rightarrow \infty}P_{n}(t)=P(t).
\end{align}In general, pointwise convergence of $P_{n}$ does not imply pointwise convergence of the corresponding Legendre transforms. Thus, to obtain the lower bound we aim for (inequality \ref{AimFOR}), we use a notion of convergence introduced by Wijsman \cite{wi} which ensures the convergence of the sequence of Legendre transforms. We say that $\{P_{n}
\}_{n\in \N}$ \textit{converges infimally} to $P$ if for every $t\in \R$ \begin{align}\label{InfimalCon}
    \lim_{\beta\rightarrow 0}\liminf_{n\rightarrow \infty}\inf\{P_{n}
(s):|t-s|<\beta\}=\lim_{\beta\rightarrow 0}\limsup_{n\rightarrow \infty}\inf\{P_{n}
(s):|t-s|<\beta\}=P(t).
\end{align} Recall that the topological pressure function $P_{n}(t)$ is non-increasing for $T_{n}$ with or without parabolic points (see \cite{gr}). Thus, for every $\beta>0$, we have that $\inf\{P_{n}(s):|t-s|<\beta\}=P(t+\beta)$. By equation \eqref{SarigConvergence} we obtain $\limsup_{n\rightarrow \infty}$ and $\liminf_{n\rightarrow \infty}$ in equation \eqref{InfimalCon} coincide for every $\beta>0$. Therefore, $\{P_{n}\}_{n\in \N}$ converges infimally to $P$.

 By Wijsman's work on infimally convergence (see \cite[Theorem 6.2]{wi2}), the fact that $\{P_{n}\}_{n\in \N}$ converges infimally to $P$ is equivalent to the fact that the respective Legendre transforms converges infimally, i.e., $\{F_{n}\}_{n\in \N}$ converges infimally to $F$ as functions on $\alpha$. Even so, the notions of punctual and infimal convergence do not imply one another.
However, when both limits exist, the function given by punctual convergence dominates the one given by infimal convergence (see \cite[Section 5]{wi2}), which implies $$\dfrac{1}{\alpha}\inf_{t\in \R}\{P(t)+t\alpha\}\leq \lim_{n\rightarrow \infty}\dfrac{1}{\alpha}\inf_{t\in \R}\{P_{n}(t)+t\alpha\}.$$ Therefore, together with inequality \eqref{Q53} we get 
\begin{align}\label{AimFOR}
    \dfrac{1}{\alpha}\inf_{t\in \R}\{P(t)+t\alpha\}\leq L(\alpha).
\end{align}We will use a covering argument on $J_{\alpha}$ to prove the other inequality; that is, we will apply directly the Hausdorff dimension definition. Without loss of generality, suppose $T$ has a parabolic fixed point at $x=0$. To inherit expansiveness conveniently, we will consider for every $m\in \N$ the set $\Lambda_{\frac{1}{m}}=\bigcap^{\infty}_{n=0}T^{-n}[\frac{1}{m},1]$. Notice that for every $r\in \N$ large enough the set $\Lambda\setminus \bigcup^{\infty}_{m=r}\Lambda_{\frac{1}{m}}$ is the set of points $x\in \Lambda$ such that all $m\in \N$, $T^{n}x\leq \frac{1}{m}$ for some $n\in \N$. Then, the set $\Lambda\setminus \bigcup^{\infty}_{m=r}\Lambda_{\frac{1}{m}}$ is the countable set of preimages of $x=0$, and so, $\Lambda$ and $\bigcup^{\infty}_{m=r}\Lambda_{\frac{1}{m}}$ have the same Hausdorff dimension since $\Lambda_{\frac{1}{m}}\subseteq \Lambda_{\frac{1}{m+1}}\subseteq \cdots \subseteq\Lambda$. In particular, 
\begin{align*}
    L(\alpha)= \text{Dim}_{H}\left(J_{\alpha}\cap \bigcup^{\infty}_{m=r}\Lambda_{\frac{1}{m}}\right).
\end{align*}Thus, it is enough to show for every $m\in \N$  that $\text{Dim}_{H}(J_{\alpha} \cap \Lambda_{\frac{1}{m}})\leq \frac{1}{\alpha}F(\alpha)$. Let $m\in \N$ be fixed. To construct an adequate cover of $J_{\alpha}\cap \Lambda_{\frac{1}{m}}$, let $t>t_{\infty}$ be a fixed parameter with the property of having a neighborhood where $P$ is decreasing. Let $\varepsilon>0$ such that $P$ is decreasing in $[t-\varepsilon,t+\varepsilon]$. By definition of $P$ respect the potential $-(t+\varepsilon)\log|T'|-P(t)$ and the linearity of the integral we have $$P_{\varepsilon}:=P(-(t+\varepsilon)\log|T'|-P(t))=P(t+\varepsilon)-P(t)<0.$$ Equation \eqref{PG1} implies there exists $M\in \N$ such that for every $k\geq M$
\begin{align}\label{Q54}
    \sum_{x:T^{k}x=x}\exp{\left(-(t+\varepsilon)S_{k}\log{|T'(x)|}-kP(t)\right)}<\exp{\left(\dfrac{P_{\varepsilon}k}{2}\right)}<1.
\end{align}Let $\delta>0$ be fixed such that $\delta\frac{P(t)}{\alpha}+\frac{P_{\varepsilon}}{2}<0$. On the other hand, for every $y\in J_{\alpha}$ there exists $N_{y}\in \N$ such that for all $k\geq N_{y}$ 
 \begin{align}\label{Q55}
     \left|\dfrac{S_{k}\log|T'(y)|}{k}-\alpha\right|<\delta.
 \end{align} 
 \begin{lemma}\label{LemmaProof1}
     There exists a constant $C_{m}>1$ such that for each $k\in \N$, there is a countable cover of $J_{\alpha}\cap \Lambda_{\frac{1}{m}}$ of cylinders with a diameter lower than $C^{-k}_{m}$.
 \end{lemma}
 \begin{proof}
     Notice each interval $\Delta_{(a_{i})^{n}_{k=0}}:=\bigcap^{n-1}_{i=0}T^{-k}([\frac{1}{m},1]\cap I_{a_{k+1}})$ holds that $$N(\Delta_{(a_{i})^{n}_{i=1}})=\min\{N_{y}:y\in J_{\alpha}\cap \Delta_{(a_{i})^{n}_{i=1}}\}<\infty,$$ since $J_{\alpha}$ is dense in $I$ by Definition \ref{MRLmaps}. For each fixed $k\in \N$, we define \begin{align*}
    \overline{A}_{k}:=\{\Delta_{(a_{i})_{i=1}^{k}}: \text{ $N(\Delta_{(a_{i})_{i=1}^{k}})\leq k$}\},
\end{align*} and for every $n> k$, we set
\begin{align*}
    A_{n}:=\{\Delta_{(a_{i})_{i=1}^{n}}: \text{ $N(\Delta _{(a_{i})_{i=1}^{n}})=n$}\}.
\end{align*}  We claim the collection of cylinders $\mathcal{A}_{k}=\overline{A_{k}}\cup \bigcup^{\infty}_{n=k+1}A_{n}$ is a countable cover of $J_{\alpha}\cap \Lambda_{\frac{1}{m}}$. Let $y\in J_{\alpha} \cap \Lambda_{\frac{1}{m}}$. Then $y\in \Delta_{(a_{i})^{N_{y}}_{i=1}}$ where $T^{j-1}y\in I_{a_{j}}$ for every $1\leq j \leq N_{y}$. Consider the cylinder $\Delta_{(a_{i})^{N}_{i=1}}$ where $N=N( \Delta_{(a_{i})^{N_{y}}_{i=1}})$ so that $\Delta_{(a_{i})^{N}_{i=1}}$ is the cylinder whose length coincides with the minimum quantity of iterations that an element in $\min( \Delta_{(a_{i})^{N_{y}}_{i=1}})$ needs for the related Birkhoff sums to be close enough to $\alpha$. Notice $N\leq N_{y}$. If $N\leq k$, then $\Delta_{(a_{i})^{k}_{i=1}}\in \overline{A}_{k}$ where we also assume $T^{j-1}y\in I_{a_{j}}$ for $1\leq j \leq k$. On the other hand, if $N>k$, then $\Delta_{(a_{i})^{N}_{i=1}}\in A_{N}$. The fact that $\mathcal{A}_{k}$ is a countable set is a consequence that the cardinality of $\N^{n}$ bounds the cardinality of $A_{n}$.

Now, we study the diameter of the elements in $\mathcal{A}_{k}$. Notice $T$ is expansive in each interval $\Delta_{n}$. By Definition \ref{MRLmaps} we have $|T'|$ is not close to 1 in these intervals ($|T|$ is constant or has polynomial growth). We define 
\begin{align*}
        C_{m}=\inf_{x\in I\setminus T^{-1}[\frac{1}{m},1]}|T'x|>1.
    \end{align*}
By the Mean value Theorem in each cylinder $\Delta_{n}$, we have that $C_{m}<|\Delta_{n}|^{-1}$. Thus, $|\Delta_{n}|<C_{m}^{-1}$. Inductively, for each array $(a_{i})^{n}_{i=1}$, we have that 
\begin{align*}
    C_{m}<\dfrac{|\Delta_{(a_{i})^{n}_{i=2}}|}{|\Delta_{(a_{i})^{n}_{i=1}}|}<\dfrac{C_{m}^{-n+1}}{|\Delta_{(a_{i})^{n}_{i=1}}|},
\end{align*} and so $|\Delta_{(a_{i})^{n}_{i=1}}|<C^{-n}_{m}$. 
 \end{proof}

The Mean value Theorem also implies that in each cylinder $\Delta_{(a_{i})^{n}_{i=1}}$, there exists $\zeta\in \Delta_{(a_{i})^{n}_{i=1}}$ such that $|\Delta_{(a_{i})^{n}_{i=1}}|^{-1}=|(T^{n})'\zeta|/(1-\frac{1}{m})$. By the tempered distortion property, for every $y\in \min(\Delta_{(a_{i})^{n}_{i=1}})$, we have that 
\begin{align}\label{IneqTDP}
|\Delta_{(a_{i})^{n}_{i=1}}|^{-1}\leq \dfrac{m}{m-1}e^{n\rho_{n}}|(T^{n})'y| 
\end{align}where $\{\rho_{n}\}_{n\in \N}$ is a sequence of positive numbers converging to zero. By definition \ref{HausdorffDim}, together with inequality \eqref{IneqTDP}, for $s=t+\varepsilon+\frac{P(t)}{\alpha}$ we get
\begin{align}\label{Q57}
    \mathcal{H}^{s}_{C^{-k}_{m}}(J_{\alpha}\cap \Lambda_{\frac{1}{m}}) \leq \sum_{\Delta_{(a_{i})^{n}_{i=1}}\in \mathcal{A}_{k}}\left(\frac{m}{m-1}e^{n\rho_{n}-S_{n}\log|T'y|}\right)^{s},
\end{align} where, abusing notation, in each addend $y\in \min(\Delta_{(a_{i})^{n}_{i=1}})$ holds inequality \eqref{IneqTDP}. Since $N_{y}\leq n$, we have 
\begin{align*}
    s(n\rho_{n}-S_{n}\log|T'y|)\leq ns\rho_{n}-(t+\varepsilon )S_{n}\log |T'y|-n\dfrac{P(t)(\alpha-\delta)}{\alpha}.
\end{align*} On the other hand, using tempered distortion on the interval $I_{(a_{i})^{n}_{i=1}}$, we have that $|(T^{n})'y|^{-1}\leq |(T^{n})x_{(a_{i})^{n}_{i=1}}|e^{n\rho_{n}}$, where $x_{(a_{i})^{n}_{i=1}}$ denotes the periodic point with period $n$ in $I_{(a_{i})^{n}_{i=1}}$. Therefore, 
\begin{align*}
    s\left(n\rho_{n}-S_{n}\log |T'(y)|\right) 
    & \leq n\rho_{n}(s+(t+\varepsilon)) -nP(t)+\delta\frac{nP(t)}{\alpha}
     \\
     & -(t+\varepsilon)S_{n}\log |T'(x_{(a_{i})^{n}_{i=1}})|.
\end{align*}Let $C_{2}=\left(\frac{m}{m-1}\right)^{s}$. Together with inequality \eqref{Q57}, we get 
\begin{align*}
    \mathcal{H}^{s}_{C^{-k}_{m}}(J_{\alpha}\cap \Lambda_{\frac{1}{m}}) \leq C_{2}\sum_{\Delta_{(a_{i})^{n}_{i=1}}\in \mathcal{A}_{k}}e^{ n\rho_{n}(s+(t+\varepsilon)) -nP(t)+\delta\frac{nP(t)}{\alpha} -(t+\varepsilon)S_{n}\log |T'(x_{(a_{i})^{n}_{i=1}})|}.
\end{align*}Since  $\frac{\delta P(t)}{\alpha}+\frac{P_{\varepsilon}}{2}<0$ there exists  $M_{2}\in \N$ such that for every $k\geq M_{2}$, we have $\rho_{k}(s+(t+\varepsilon))+\frac{1}{2}\left(\frac{\delta P(t)}{\alpha}+\frac{P_{\varepsilon}}{2}\right)<0$ since $\{\rho_{n}\}\rightarrow 0$. Let $\gamma=\sup_{k\geq M_{2}}\rho_{k}(s+(t+\varepsilon))$. Separating the cylinders by length, for every $k\geq M_{2}$ 
\begin{align*}
    \mathcal{H}^{s}_{C^{-k}_{m}}(J_{\alpha}\cap \Lambda_{\frac{1}{m}}) \leq C_{2}\sum^{\infty}_{n=k}e^{n(\gamma+\frac{\delta P(t)}{\alpha})} \sum_{ \Delta_{(a_{i})^{n}_{i=1}}\in \mathcal{A}_{k}}e^{-n\left(S_{n}(t+\varepsilon)\log |T'(x_{(a_{i})^{n}_{i=1}})|+P(t)\right)}.
\end{align*}The sums in the right-hand side of the inequality are supported on periodic points of $T$, we have for $k\geq \max\{M,M_{2}\}$ (see inequality \eqref{Q54})
\begin{align*}
    \mathcal{H}^{s}_{C_{m}^{-k}}\left(J_{\alpha}\cap \Lambda_{\frac{1}{m}}\right) 
    & \leq C_{2} \sum^{\infty}_{n=k}e^{n(\gamma+\delta\frac{P(t)}{\alpha})} \sum_{x:T^{r}x=x}e^{-n\left(S_{n}(t+\varepsilon)\log|T'(x)|+P(t)\right)}\\
    & \leq C_{2} \sum^{\infty}_{n=k}e^{ n\left(\gamma+\delta\frac{P(t)}{\alpha}+\frac{P_{\varepsilon}}{2}\right)}.
\end{align*}
 Since $\gamma+\delta\frac{P(t)}{\alpha}+\frac{P_{\varepsilon}}{2}<0$, we get that $\sum^{\infty}_{n=r}e^{ n\left(\gamma+\delta\frac{P(t)}{\alpha}+\frac{P_{\varepsilon}}{2}\right)}$ is a geometric series. Thus, because both the tail of the series and $C^{-k}_{m}$ converge to zero when $k\rightarrow \infty$, 
\begin{align*}
   \mathcal{H}^{s}\left(J_{\alpha}\cap \Lambda_{\frac{1}{m}}\right) =\lim_{k\rightarrow \infty} \mathcal{H}^{s}_{C_{m}^{-k}}\left(J_{\alpha}\cap \Lambda_{\frac{1}{m}}\right)= 0.
\end{align*} Thus, $\text{Dim}_{H}(J_{\alpha}\cap \Lambda_{\frac{1}{m}})\leq (t+\varepsilon)+\frac{P(t)}{\alpha}$. Since $\varepsilon>0$ was chosen arbitrarily, $\text{Dim}_{H}(J_{\alpha}\cap \Lambda_{\frac{1}{m}})\leq t+\frac{P(t)}{\alpha}$ for all $t>t_{\infty}$ having a neighborhood where $P$ is strictly decreasing. On the other hand, $P$ is a slide function by Theorem \ref{Thm 1}. Therefore, the same arguments of differentiability used in Proposition \ref{Lemma 1} imply $\text{Dim}_{H}(J_{\alpha}\cap \Lambda_{\frac{1}{m}})\leq \frac{1}{\alpha}F(\alpha)$, where $F(\alpha)=\inf_{t\in \R}\{P(t)+t\alpha\}$. This completes the proof of equation \eqref{Q51} for $\alpha\neq 0$. On the other hand, equation \eqref{Q51} and Proposition \ref{Lemma 1} imply the domain of $L$ to be an unbounded interval of non-negative numbers and that $L(0)=\lim_{\alpha \rightarrow 0+}L(\alpha)$ exists in case $0\in Dom(L)$.
\end{proof}

\begin{remark}
    It is worth noting that, in the finitely many branches scenario $(T_{n},\Lambda_{n})$, the derivative of the pressure $P_{n}$ at $t$ coincides with $-\int \log{|T_{n}'|}d\mu_{t}$ (see \cite[Theorem 5.6.5.]{pu}), where $\mu_{t}$ is the equilibrium state for $-t\log{|T_{n}'|}$. Thus, for $\mu_{t}$ almost every point $x\in \Lambda_{n}$ we have $\lambda(x)=-P_{n}'(t)$.
\end{remark}
 \begin{figure}
    \centering
    \includegraphics[width=0.7
    \linewidth]{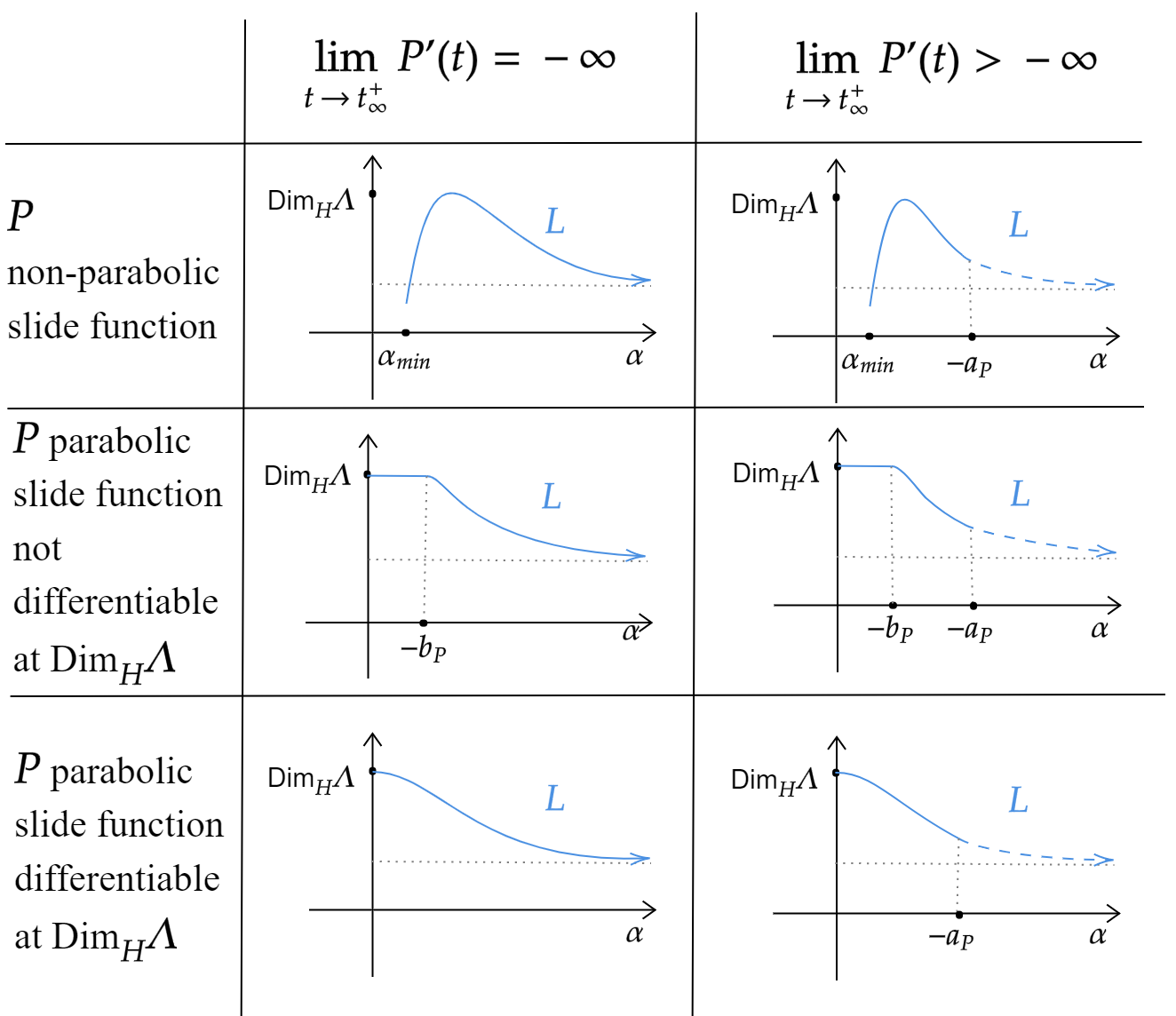}
    \caption{The Lyapunov spectrum graph of $MRL$ maps is depicted, where $a_P$ represents the limit as $t$ approaches $t_{\infty}^{+}$ of $P'(t)$ and $b_P$ represents the limit as $t$ approaches $\text{Dim}_{H}^{-}\Lambda$ of $P'(t)$. Additionally, for some $\alpha_{\text{min}} \geq 0$, we have $Dom(L) \subseteq [\alpha_{\text{min}},\infty)$. When $\alpha \geq -a_P$, the dashed segments correspond to $L(\alpha) = t_{\infty} + \frac{\lim_{t\rightarrow t_{\infty}^{+}}P(t)}{\alpha}$. The intersection of the horizontal dotted line with the $y$-axis represents the Hausdorff dimension of the set of points for which the Lyapunov exponent is not defined.}
    \label{fig:my_label3}
\end{figure}

In Figure \ref{fig:my_label3} we present different behaviors of the Lyapunov spectrum given by MRL maps. The following are two examples that have not been considered before in which Theorem \ref{THM 5.1} applies. 
\begin{example}\normalfont\label{Example 5.3}
Let $m\in\N$ be fixed. Let $C_{m}$,and let the sequence $\{a_{n,m}\}$ be the same series and sequence as in example \ref{Example 4.5}. By differentiating the pressure function $P_{m}(t)=\log\left(\sum^{\infty}_{n=1}a_{n,m}^{t}\right)$, we obtain for all $t>t_{\infty}$ $$P_{m}'(t)=\dfrac{\sum^{\infty}_{n=1}a_{n,m}^{t}\log(a_{n,m})}{\sum^{\infty}_{n=1}a_{n,m}^{t}},$$ that is, for $D_{m}:=((n+10)\log^{m}(n+10))^{2t}$ we get
    \begin{align*}
        P_{m}'(t)=-\log(C_{m})-\dfrac{2}{\sum^{\infty}_{n=1}D_{m}}\left(\sum_{n=1}^{\infty}\dfrac{\log(n+10)}{D_{m}}+m\sum_{n=1}^{\infty}\dfrac{\log(\log(n+10))}{D_{m}}\right).
    \end{align*} The first series in the parentheses is a Dirichlet series, and the second is dominated by the first. Using \cite[Theorem 1]{gt}, for $m=1$, the series $\sum_{n=1}^{\infty}\frac{\log(n)}{(n+10)\log^{1}(n+10))^{2t}}$ diverges as $t\rightarrow t_{\infty}^{+}$. Therefore $\lim_{t\rightarrow t_{\infty}^{+}}P_{1}'(t)=-\infty$. On the other hand, if we take $m=5$, we will have by \cite[Theorem 1]{gt} that the first series of parentheses converges when $t\rightarrow t_{\infty}^{+}$, therefore $\lim_{t\rightarrow t_{\infty}^{+}}P_{5}'(t)>-\infty$. So, just as $L_{\{a_{n,m}\}_{n\in \N}}$ are MRL maps, in Figure \ref{fig:my_label3}, the Lyapunov spectrum graphs of $L_{\{a_{n,1}\}_{n\in \N}}$ 
 and $L_{\{a_{n,5}\}_{n\in \N}}$ are similar to the first and second graphs of the first row, respectively. Analogous examples have been shown in \cite[Figure 5.3. and Figure 5.4]{KMS}.
\end{example}

\begin{example}\normalfont\label{Example 5.4}
   Consider $T_{m}$ the infinite Manneville Pomeu map of example \ref{Example 4.6}, i.e., let $\frac{1}{2}\leq s \leq 1$ be fixed and let $c\in \R$ be such that $1=c+c^{1+s}$. Bet $m\in \N$ be fixed and let the sequence $\{(1-c)a_{n,m}\}_{n\in \N}$ where $\{a_{n,m}\}_{n\in N}$ is the same sequence as in example \ref{Example 5.3}. By Theorem \ref{Thm 1} and \cite[Theorem 1]{gt}, the pressure function $P_{5}$ with respect to $T_{5}$ is a parabolic slide function of discontinuous type. Then, the graph of the Lyapunov spectrum of $T_{5}$ is similar to one of the graphs in the second row of Figure \ref{fig:my_label3}.  
\end{example}
By the definition of the Legendre transform of the topological pressure function and the Theorem \ref{THM 5.1}, we obtain the following corollary

\begin{corollary}\label{Cor53}
Let $T$ be an MRL map. For every $\alpha\in Dom(L)$  $$F(\alpha)=\alpha L(\alpha).$$
\end{corollary}
\begin{proof}\textit{of theorem \ref{PrincipalT}.} By Theorem \ref{Thm 1}, the topological pressure function $P$ is a slide function with $d=\text{Dim}_{H}(\Lambda)$. Thus, using both Corollaries \ref{Lemma2} and \ref{Cor53}, we obtain $$L(\alpha)=\dfrac{1}{\alpha}F(\alpha)=Ns_{P}(\alpha).$$  \end{proof}
\begin{remark}\label{GeoRemark}
Geometrically, the theorem says that for any support line tangent to the graph of $P$ with slope $-\alpha$ its intersections with the coordinate axes $x$ and $y$ are $L(\alpha)$ and $F(\alpha)$ respectively.
\end{remark}
    This same result holds for parabolic dynamical systems with finitely many branches since an analogous result to Theorem \ref{THM 5.1} is shown in \cite{gr}. We can prove it analogously to theorem \ref{PrincipalT} by defining the S-Newton-Raphson map for $P$, i.e., $Ns_{P}(\alpha)=N_{P}\circ (P')^{-1}(\alpha)$ whenever $-\alpha<0$ is in the image of $P'$, or $Ns_{P}(\alpha)=\text{Dim}_{H}(\Lambda)$ whenever $T$ has parabolic fixed points and $\lim_{t\rightarrow \text{Dim}_{H}(\Lambda)^{-}}P'(t)\leq -\alpha< 0$. For example, Theorem \ref{PrincipalT} holds for the Mannevile Pomeau map.

\section*{Acknowledments}
I want to thank my Ph.D. advisor Godofredo Iommi for suggesting this problem, for his patience, and for numerous helpful remarks. Also, thank the referees for the suggestions in the reviewed version of this document. This research was supported by ANID Doctorado Nacional 21210037.

\end{document}